\documentclass[10pt]{amsart}
\usepackage[all]{xy}
\usepackage[backref=page]{hyperref}

\begin{document}

\newcommand{\V}{{\cal V}}      
\renewcommand{\O}{{\cal O}}
\newcommand{\LL}{\cal L}
\newcommand{\Ext}{\hbox{{\rm Ext}}}
\newcommand{\depth}{\hbox{{\rm depth}}}
\newcommand{\pd}{{\hbox{{\rm pd}}}}
\newcommand{\Tor}{\hbox{Tor}}
\newcommand{\Hom}{\hbox{Hom}}
\newcommand{\Proj}{\hbox{Proj}}
\newcommand{\GrMod}{\hbox{GrMod}}
\newcommand{\grmod}{\hbox{gr-mod}}
\newcommand{\tors}{\hbox{tors}}
\newcommand{\rank}{\hbox{{\rm rank}}}
\newcommand{\End}{\hbox{End}}
\newcommand{\GKdim}{\hbox{GKdim}}
\newcommand{\im}{\hbox{im}}
\renewcommand{\ker}{\hbox{ker}}
\newcommand{\coker}{\ensuremath{\operatorname{coker}}}
\newcommand{\isom}{\cong}
\newcommand{\lk}{\text{link}}
\newcommand{\soc}{\text{soc}}

\newcommand{\lonto}{{\protect \longrightarrow\!\!\!\!\!\!\!\!\longrightarrow}}

\newcommand{\m}{{\mu}}
\newcommand{\gl}{{\frak g}{\frak l}}
\newcommand{\ssl}{{\frak s}{\frak l}}
\newcommand{\I}{\mathfrak{I}}
\newcommand{\A}{\mathfrak{A}}
\newcommand{\G}{\mathfrak{G}}

\newcommand{\ds}{\displaystyle}
\newcommand{\s}{\sigma}
\renewcommand{\l}{\lambda}
\renewcommand{\a}{\alpha}
\renewcommand{\b}{\beta}
\newcommand{\g}{\gamma}
\newcommand{\z}{\zeta}
\newcommand{\e}{\varepsilon}
\newcommand{\D}{\Delta}
\renewcommand{\d}{\delta}
\newcommand{\p}{\rho}
\renewcommand{\t}{\tau}

\newcommand{\C}{{\mathbb C}}
\newcommand{\N}{{\mathbb N}}
\newcommand{\Z}{{\mathbb Z}}
\newcommand{\ZZ}{{\mathbb Z}}
\newcommand{\K}{{\mathcal K}}
\newcommand{\F}{{\mathcal F}}
\newcommand{\cS}{\mathcal S}

\newcommand{\rowxy}{(x\ y)}
\newcommand{\colxy}{ \left({\begin{array}{c} x \\ y \end{array}}\right)}
\newcommand{\scolxy}{\left({\begin{smallmatrix} x \\ y
\end{smallmatrix}}\right)}

\renewcommand{\P}{{\Bbb P}}

\newcommand{\la}{\langle}
\newcommand{\ra}{\rangle}
\newcommand{\tensor}{\otimes}

\newtheorem{introthm}{Theorem}
\newtheorem{thm}{Theorem}[section]
\newtheorem{lemma}[thm]{Lemma}
\newtheorem{cor}[thm]{Corollary}
\newtheorem{prop}[thm]{Proposition}

\renewcommand{\theintrothm}{\Alph{introthm}}

\theoremstyle{definition}
\newtheorem{defn}[thm]{Definition}
\newtheorem{notn}[thm]{Notation}
\newtheorem{ex}[thm]{Example}
\newtheorem{rmk}[thm]{Remark}
\newtheorem{rmks}[thm]{Remarks}
\newtheorem{note}[thm]{Note}
\newtheorem{example}[thm]{Example}
\newtheorem{problem}[thm]{Problem}
\newtheorem{ques}[thm]{Question}
\newtheorem{conj}[thm]{Conjecture}
\newtheorem{cons}[thm]{Construction}
\newtheorem{convention}[thm]{Convention}
\newtheorem*{property*}{Property (F)}
\newtheorem{thingy}[thm]{}

\theoremstyle{remark}
\newtheorem*{rmk*}{Remark}

\newcommand{\onto}{{\protect \rightarrow\!\!\!\!\!\rightarrow}}
\newcommand{\donto}{\put(0,-2){$|$}\put(-1.3,-12){$\downarrow$}{\put(-1.3,-14.5) 

{$\downarrow$}}}

\newcounter{letter}
\renewcommand{\theletter}{\rom{(}\alph{letter}\rom{)}}

\newenvironment{lcase}{\begin{list}{~~~~\theletter} {\usecounter{letter}
\setlength{\labelwidth4ex}{\leftmargin6ex}}}{\end{list}}

\newcounter{rnum}
\renewcommand{\thernum}{\rom{(}\roman{rnum}\rom{)}}

\newenvironment{lnum}{\begin{list}{~~~~\thernum}{\usecounter{rnum}
\setlength{\labelwidth4ex}{\leftmargin6ex}}}{\end{list}}


\title[Twisted matrix factorizations]{Periodic free resolutions from twisted matrix factorizations}



\author[  Cassidy, Conner, Kirkman, Moore ]{ }

\begin{abstract}
  The notion of a matrix factorization was introduced by Eisenbud in
  the commutative case in his study of bounded (periodic) free
  resolutions over complete intersections.  Since then, matrix
  factorizations have appeared in a number of applications.  In this work,
  we extend the notion of (homogeneous) matrix factorizations to
  regular normal elements of connected graded algebras over a field.

  Next, we relate the category of twisted matrix factorizations to an
  element over a ring and certain Zhang twists.  We also show that in
  the AS-regular setting, every sufficiently high syzygy module is the
  cokernel of some twisted matrix factorization.  Furthermore, we show
  that in this setting there is an equivalence of categories between
  the homotopy category of twisted matrix factorizations and the
  singularity category of the hypersurface, following work of Orlov.
\end{abstract}

\maketitle

\begin{center}

Thomas Cassidy \\
\bigskip

Department of Mathematics\\ 
Bucknell University\\
Lewisburg, PA 17837\\
\bigskip
\bigskip

\vskip-.1in
Andrew Conner \\
Ellen Kirkman\\
W. Frank Moore\\
\bigskip

Department of Mathematics\\
Wake Forest University\\
Winston-Salem, NC 27109\\

\end{center}

\setcounter{page}{1}

\thispagestyle{empty}

\vspace{0.2in}

\bigskip

\section*{Introduction}

The notion of a matrix factorization was introduced by Eisenbud
\cite{Eis} in the commutative case in his study of bounded (periodic)
free resolutions over complete intersections.  Since then, matrix
factorizations have appeared in a number of applications, including string
theory \cite{Asp,KapLi}, singularity categories \cite{Orlov2, Orlov1},
representation theory of Cohen-Macaulay modules \cite{BuGS,Kn}, and
other topics \cite{DyckMurf,PolVain}.  Recently, Eisenbud and Peeva
also extended the notion of matrix factorizations to higher
codimension complete intersections \cite{EisPeev}.

There are several candidates for the notion of a complete intersection
in the case of a non-commutative algebra, based on their numerous
characterizations in the commutative case \cite{FT,FHT,Gull1,Gull2}.
One possible approach is to first understand the hypersurface case by
studying matrix factorizations in the non-commutative setting.  In
this work, we extend the notion of (homogeneous) matrix factorizations
to regular normal elements of connected graded algebras over a field.

To state our results below, we assume that $A$ is a connected,
$\N$-graded, locally finite dimensional algebra over a field $k$.  We
also fix a homogeneous normal regular element $f \in A_+ =
\bigoplus_{n > 0} A_n$ and set $B = A / (f)$.  The regularity and
normality of $f$ provide us with a graded automorphism $\sigma$ of $A$
which we incorporate into the definition of matrix factorization.  The
use of $\sigma$ to modify ring actions is the reason for our
``twisted'' terminology; see Definition \ref{defn:MF}.

Our first main result shows that just as in the commutative case,
(reduced) twisted matrix factorizations give rise to (minimal)
resolutions.

\begin{introthm}[Proposition \ref{exactness}] \label{thmA} A twisted
  left matrix factorization $(\varphi,\tau)$ of $f$ gives rise to a complex
  $\mathbf{\Omega}(\varphi,\tau)$ of free left $B$-modules which is a graded free
  resolution of $\coker \varphi$ as a left $B$-module.  If the twisted matrix
  factorization $(\varphi,\tau)$ is reduced (see Definition
  \ref{defn:redMF}), then the graded free resolution is minimal.  If
  the order of $\sigma$ is finite, then the resolution is periodic of
  period at most twice the order of $\sigma$.
\end{introthm}

As in the commutative case, one can consider the category of all
twisted matrix factorizations of $f$ over a ring $A$, which we call
$TMF_A(f)$.

There is another context where twisting via an automorphism arises in
the study of graded algebras: the Zhang twist \cite{Zhang}.  The
following theorem relates the category of twisted matrix
factorizations of $f$ over $A$ to those over the Zhang twist of $A$
with respect to a compatible twisting system $\zeta$ (which we denote $A^\zeta$).

\begin{introthm}[see Theorem \ref{twistEquiv}] \label{thmB} Let $\zeta
  = \{\sigma^n~|~n \in \mathbb{Z}\}$ be the twisting system associated
  with the normalizing automorphism $\sigma$.  Then the categories
  $TMF_A(f)$ and $TMF_{A^\zeta}(f)$ are equivalent.
\end{introthm}

This result is somewhat surprising. If $f$ is central in the Zhang
twist $A^\zeta$, then the complexes associated to matrix
factorizations in $TMF_{A^\zeta}(f)$ will be periodic of period at
most two, while those coming from matrix factorizations in $TMF_A(f)$
could have a longer period, depending on the order of $\sigma$.  It
should be noted that $f$ is not necessarily central in the
Zhang twist.  This peculiarity is illustrated in Example \ref{ex:ore}.

A major result in \cite{Eis} that drives many of the applications of
matrix factorizations is that, under appropriate hypotheses, every
minimal graded free resolution is eventually given by a reduced matrix
factorization.  Using J\o rgensen's version of the Auslander-Buchsbaum
formula for connected graded algebras\cite{Jorg}, we are able to
extend this result:

\begin{introthm}[Theorem \ref{mainThm}] \label{thmC} Let $A$ be a left
  noetherian Artin-Schelter regular algebra of dimension $d$.  Let $f
  \in A_+$ be a homogeneous normal regular element and let $B = A/(f)$.
  Then for every finitely generated graded left $B$-module $M$, the
  $(d+1)^\text{st}$ left syzygy of $M$ is the cokernel of some reduced
  twisted left matrix factorization of $f$ over $A$.
\end{introthm}

There is a suitable notion of homotopy in the category $TMF_A(f)$, and
we denote the associated homotopy category $hTMF_A(f)$.  Following
Orlov's lead \cite{Orlov1}, we provide a triangulated structure on $hTMF_A(f)$ and
prove the following Theorem.  

\begin{introthm}[Theorem \ref{catEquiv}] \label{thmD} Let $A$ be a
  left noetherian Artin-Schelter regular algebra, and let $f \in A_+$
  be a homogeneous normal regular element.  Then the homotopy category
  of twisted matrix factorizations of $f$ over $A$ is equivalent to
  the bounded singularity category of $B$.
\end{introthm}

It should be noted that since the minimal resolution that comes from a
twisted matrix factorization need not be periodic, some minor adjustments to
Orlov's original argument must be made.

The paper is organized as follows: Section \ref{section:Preliminaries}
covers preliminaries, as well as sets up notation regarding various
twists that will be used for the remainder of the paper.  Section
\ref{section:MF} covers the definition of matrix factorization, as
well as the proof of Theorem \ref{thmA}.  Section \ref{section:zhang}
contains the background regarding the Zhang functor, as well as the
proof of Theorem \ref{thmB}.  Section \ref{section:hypersurfaces}
includes the background on the Auslander-Buchsbaum theorem in our
setting due to J\o rgensen, as well as the precise statement and proof
of Theorem \ref{thmC}.  Section \ref{section:categories} contains the
categorical considerations for Theorem \ref{thmD}, and Section
\ref{section:examples} contains some examples.

\section*{Acknowledgements}
Many computations were performed using the \texttt{NCAlgebra} package
written by Conner and Moore for Macaulay2 \cite{M2}; it provides a
Macaulay2 interface to the Bergman \cite{Bergman} system and was
invaluable in performing some of the calculations.

\section{Preliminaries}
\label{section:Preliminaries}
The main results in this paper concern graded modules over graded
rings, hence we work exclusively in that context.  Let $A$ be a
connected, $\N$-graded algebra over a field $k$. We assume $A$ is
locally finite-dimensional: $\dim_k A_i<\infty$ for all
$i\in\N$. Throughout, we work in the category $A$-GrMod of graded left
$A$-modules with degree 0 morphisms, though our definitions and
results have obvious analogs for graded right modules. Let $\sigma$ be
a degree 0 graded algebra automorphism of $A$. For $M\in
A\text{-GrMod}$, we write $M^{\sigma}$ for the graded left $A$-module
with $M^{\sigma}=M$ as graded abelian groups and left $A$-action
$a\cdot m= \sigma(a)m$. If $\varphi:M\rightarrow N$ is a degree 0
homomorphism of graded left $A$-modules,
$\varphi^{\sigma}=\varphi:M^{\sigma}\rightarrow N^{\sigma}$ is also a
graded module homomorphism. It is straightforward to check that the
functor $(-)^{\sigma}$ is an autoequivalence of $A$-GrMod.

For any $n\in\Z$ and $M\in A$-GrMod, we write $M(n)$ for the shifted
module whose degree $i$ component is $M(n)_i=M_{i+n}$. The degree
shift functor $M\mapsto M(n)$ is also easily seen to be an
autoequivalence of $A$-GrMod which commutes with $(-)^{\sigma}$; that
is, $M(n)^{\sigma}=M^{\sigma}(n)$.

Let $f\in A_d$ be a normal, regular homogeneous element of degree $d$,
and let $\sigma:A\rightarrow A$ be the graded automorphism of $A$
determined by the equation $af=f\sigma(a)$ for each $a\in A$. We call
$\sigma$ the \emph{normalizing automorphism} of $f$ and say $f$
\emph{is normalized by} $\sigma$. Note that $f$ is normalized by
$\sigma$ if and only if left multiplication by $f$ is a graded left
module homomorphism $\lambda_f^M:M^{\sigma}(-d)\rightarrow M$ for all
$M$. Moreover, $\lambda_f^N\varphi^{\sigma}(-d)=\varphi\lambda_f^M$
for any graded homomorphism $\varphi:M\rightarrow N$. The composite
functor $M\mapsto M^{\sigma}(-d)$ will be used so frequently for fixed
$\sigma$ and $d$ that we define $M^{tw}=M^{\sigma}(-d)$ and
$M^{tw^{-1}}=M^{\sigma^{-1}}(d)$.

In this paper we are especially interested in periodic resolutions. We
say a degree 0 complex $\mathbf{P}:\cdots\rightarrow P_2\rightarrow
P_1\rightarrow P_0$ of graded left $A$-modules is \emph{periodic of period $p$}
if $p$ is the smallest positive integer such that there exists an integer
$n$ and a morphism of complexes $t:\mathbf{P}(n)\rightarrow \mathbf{P}$
of (homological) degree $-p$ where $t:P_{i+p}(n)\rightarrow P_i$ is an isomorphism
for all $i\ge 0$. Note the shift $(n)$ is applied to the internal grading of each
free module in the complex. If such an integer $p$ exists, we say
$\mathbf{P}$ is \emph{periodic}.  If there exists an integer $m\ge 0$
such that the truncated complex $\cdots\rightarrow P_{m+2}\rightarrow
P_{m+1}\rightarrow P_{m}$ is periodic, we say $\mathbf{P}$ is
\emph{periodic after $m$ steps}.

We record a few straightforward facts about periodic complexes needed later. 

\begin{lemma}
\label{resFacts}
Let $\mathbf{P}$ be a complex of graded free left $A$-modules. 
\begin{enumerate}
\item If $\mathbf{P}$ is periodic and there exists an integer $N>0$
  such that $\rank\ P_j=\rank\ P_N$ for all $j\ge N$, then $\rank\
  P_j=\rank\ P_0$ for all $j\ge 0$.
\item Let $\widetilde{\mathbf{P}}$ be a complex of graded free left
  $A$-modules which is periodic of period $p$. If there exist an
  integer $n$ and an isomorphism of complexes
  $t:\widetilde{\mathbf{P}}(n)\rightarrow \mathbf{P}$, then
  $\mathbf{P}$ is also periodic of period $p$.
\end{enumerate}
\end{lemma}

We say a resolution $(\mathbf{P}_{\bullet},d_{\bullet})$ is minimal if
$\im\ d_i\subset A_+P_{i-1}$ for all $i$, where $A_+=\bigoplus_{n>0}
A_n$.  Recall that every bounded below, graded module over a
connected, $\N$-graded, locally finite-dimensional $k$-algebra has a
minimal graded free resolution. This resolution is unique up to
non-unique isomorphism of complexes (see, for example, \cite{PP}).

\begin{lemma}
\label{stability}
Let $\sigma$ be a degree 0 graded automorphism of $A$. Let
$\mathbf{P}$ be a minimal graded free resolution of a bounded below,
graded left $A$-module $M$. If $M^{\sigma}\cong M$ as graded modules,
then the chain complexes $\mathbf{P}^{\sigma}$ and $\mathbf{P}$ are
isomorphic.
\end{lemma}

\begin{proof}
  First note that $\mathbf{P}^{\sigma}$ is a minimal graded free
  resolution of $M^{\sigma}$.  Let $\psi:M^{\sigma}\rightarrow M$ be a
  graded isomorphism. Let $\Psi:\mathbf{P}^{\sigma}\rightarrow
  \mathbf{P}$ and $\Phi:\mathbf{P}\rightarrow \mathbf{P}^{\sigma}$ be
  the graded morphisms of complexes obtained by lifting $\psi$ and
  $\psi^{-1}$ respectively. By the Comparison Theorem, $\Psi\Phi$ and
  $\Phi\Psi$ are homotopic to the respective identity maps. Since the
  complexes $\mathbf{P}$ and $\mathbf{P}^{\sigma}$ are minimal,
  $\Psi\Phi$ and $\Phi\Psi$ are isomorphisms. Hence $\Psi$ and $\Phi$
  are isomorphisms.
\end{proof}

\section{Twisted matrix factorizations}
\label{section:MF}
The key to our study is the notion of a twisted matrix
factorization. As will be evident, the notion can in fact be defined
over any ring containing a normal, regular element. Indeed some
results, such as Proposition \ref{exactness}, are readily seen to hold
in this generality by forgetting the grading.

As above, let $A$ be a connected, $\N$-graded, locally
finite-dimensional algebra over a field $k$. Let $f\in A_d$ be a
normal, regular homogeneous element of degree $d$ and let $\sigma$ be
its degree 0 normalizing automorphism. In this section we do not
require $A$ to be Artin-Schelter regular.

\begin{defn}
\label{defn:MF}
A \emph{twisted left matrix factorization} of $f$ over $A$ is an
ordered pair of maps of finitely generated graded free left
$A$-modules $(\varphi:F\rightarrow G, \tau:G^{tw}\rightarrow F)$ such
that $\varphi\tau=\lambda_f^G$ and $\tau\varphi^{tw}=\lambda_f^F$.
\end{defn}

Our definition is a naive generalization of the familiar notion from
commutative algebra. We remark that the freeness of $G^{tw}$ requires
$\sigma$ to be a graded automorphism. Also note that $(\varphi,\tau)$
is a twisted matrix factorization if and only if either
$(\varphi^{tw},\tau^{tw})$ or $(\tau,\varphi^{tw})$ is.  It is easy to see
that if $(\varphi,\tau)$ is a twisted matrix factorization, then both
$\varphi$ and $\tau$ are injective since $f$ is regular.

Free modules over noncommutative rings need not have a well-defined notion of rank. Even among those that do, not all satisfy the \emph{rank conditions} (a) $f:A^n\rightarrow A^m$ an epimorphism $\Rightarrow n\ge m$ and (b)  $f:A^n\rightarrow A^m$ a monomorphism $\Rightarrow n\le m$ (though (b) implies (a), see \cite{Lam}). However, since we assume $A$ is locally finite dimensional and morphisms preserve degree, the graded version of (b) clearly holds for graded free $A$-modules. Thus rank is well-defined for graded free $A$-modules. As noted above, if $(\varphi:F\rightarrow G, \tau:G^{tw}\rightarrow F)$ is a twisted
  left matrix factorization, then $\varphi$ and $\tau$ are injective. It follows that $\rank\ F=\rank\ G$.

For categorical reasons (see below) we adopt the usual convention that the
zero module is free on the empty set. We call the twisted
factorization $(\varphi,\tau)$ where $\varphi=\tau:0\rightarrow 0$ the
\emph{irrelevant factorization}.  We call a twisted factorization
$(\varphi,\tau)$ \emph{trivial} if $\varphi=\lambda_f^A$ or
$\tau=\lambda_f^A$.

Paralleling the commutative case, twisted matrix factorizations
provide a general construction of resolutions.

\begin{prop}
\label{exactness}
Let $(\varphi:F\rightarrow G, \tau:G^{tw}\rightarrow F)$ be a twisted
left matrix factorization of a regular element $f\in A$ with
normalizing automorphism $\sigma$. Set $B=A/(f)$ and write
$\overline{{\ }^{} }$ for reduction modulo $f$. Then the complex
$$\mathbf{\Omega}(\varphi,\tau) : \cdots\rightarrow \overline{G}^{\overline{tw}^2}\xrightarrow{\overline{{\tau}}^{\overline{tw}}}
                                                    \overline{F}^{\overline{tw}}\xrightarrow{\overline{{\varphi}}^{\overline{tw}}}
                                                    \overline{G}^{\overline{tw}}\xrightarrow{\overline{\tau}}
                                                    \overline{F}\xrightarrow{\overline{\varphi}} \overline{G}$$
is a resolution of $M={\rm coker\ \varphi}$ by free left $B$-modules.
\end{prop}

Before giving the proof, we must address a potential source of
confusion stemming from the fact that $M$ is both a left $A$-module
and a left $B$-module. Since $\sigma(f)=f$, $\sigma$ induces an
automorphism $\bar{\sigma}:B\rightarrow B$. We will write $M^{tw}$ to
mean $(_AM)^{\sigma}(-d)$ and $M^{\overline{tw}}$ to mean
$(_BM)^{\bar{\sigma}}(-d)$. The reader should note that
$M^{\overline{tw}}\isom B\tensor_A M^{tw}$ and likewise
$_A(M^{\overline{tw}})\isom M^{tw}$. We denote repeated application of
of $(-)^{tw}$ by $(-)^{tw^2}$, $(-)^{tw^3}$, etc.\ and likewise for
$(-)^{\overline{tw}}$.

The proof of Proposition \ref{exactness} is a straightforward
generalization of the commutative case.

\begin{proof}
  Since $\varphi\tau=\lambda_f^G$, we see that $f({\rm coker\ }
  \varphi)=0$ so ${\rm coker\ } \overline{\varphi} = {\rm coker\
  }\varphi$.

  We prove exactness at $\overline{F}^{\overline{tw}^i}$, exactness at
  $\overline{G}^{\overline{tw}^i}$ being analogous. Let $K$ be a
  graded free $A$-module and $\kappa:K\rightarrow F^{tw^i}$ an
  $A$-module map such that $\overline{\kappa}$ is a $B$-module
  surjection onto $\ker\ \overline{\varphi}^{\overline{tw}^i}$. Then
  $\im\ \varphi^{tw^i}\kappa\subseteq f G^{tw^i}$ and we can define an
  $A$-module map $h:K\rightarrow
  ({G}^{tw^i})^{{}^{tw}}={G}^{tw^{i+1}}$ by $h({x})={g}$ where $g\in
  G^{tw^i}$ satisfies $\varphi^{tw^i}\kappa(x)=fg$. For any $a\in A$
  we have $\varphi^{tw^i}\kappa(ax)=afg=f\sigma(a)g$ so $h$ is
  $A$-linear.  Since $f$ is regular, $\overline{h}$ does not depend on
  the choice of $g$. We have
$$\lambda^{F^{tw^i}}_f\kappa = \tau^{tw^{i-1}}\varphi^{tw^i}\kappa=\tau^{tw^{i-1}}\lambda^{G^{tw^i}}_f h=\lambda^{F^{tw^i}}_f\tau^{tw^i} h$$
Again, since $f$ is regular,
$\overline{\kappa}=\overline{\tau}^{\overline{tw}^i}\overline{h}$,
hence $\ker\ \overline{\varphi}^{\overline{tw}^i}=\im\
\overline{\kappa}\subseteq \im\ \overline{\tau}^{\overline{tw}^i}$.
\end{proof}

In light of Proposition \ref{exactness}, we make the following natural definition.  

\begin{defn}
  A morphism $(\varphi,\tau)\rightarrow (\varphi', \tau')$ of twisted
  left matrix factorizations of $f$ over $A$ is a pair
  $\Psi=(\Psi_G,\Psi_F)$ of degree 0 module homomorphisms
  $\Psi_G:G\rightarrow G'$ and $\Psi_F:F\rightarrow F'$ such that the
  following diagram commutes.
$$
\xymatrix{
F\ar@{->}[r]^{\varphi}\ar@{->}[d]_{\Psi_F} &G\ar@{->}[d]^{\Psi_G}\\
F'\ar@{->}[r]_{\varphi'} &G'\\
}
$$
A morphism $\Psi$ is an \emph{isomorphism} if $\Psi_G$ and $\Psi_F$ are isomorphisms. 
\end{defn}

The regularity of $f$ guarantees that
$(\Psi_G,\Psi_F):(\varphi,\tau)\rightarrow (\varphi', \tau')$ is a
morphism if and only
if $$(\Psi_F,\Psi_G^{tw}):(\tau,\varphi^{tw})\rightarrow (
\tau',(\varphi')^{tw})$$ is. It is clear that $(\Psi_G,\Psi_F)$ is a
morphism if and only
if $$(\Psi_G^{tw},\Psi_F^{tw}):(\varphi^{tw},\tau^{tw})\rightarrow
((\varphi')^{tw}, (\tau')^{tw})$$ is.  We leave the straightforward
proof of the next Proposition to the reader.

\begin{prop}
\label{morphisms}
A morphism $\Psi:(\varphi,\tau)\rightarrow(\varphi',\tau')$ of twisted
matrix factorizations of $f$ over $A$ induces a chain morphism of
complexes $\mathbf{\Omega}(\Psi):\mathbf{\Omega}(\varphi,\tau)\rightarrow
\mathbf{\Omega}(\varphi',\tau')$.  Twisted matrix factorizations
$(\varphi,\tau)$ and $(\varphi', \tau')$ are isomorphic if and only if
the complexes $\mathbf{\Omega}(\varphi,\tau)$ and
$\mathbf{\Omega}(\varphi',\tau')$ are chain isomorphic.
\end{prop}

\begin{defn}
\label{defn:redMF}
We define the \emph{direct sum} of twisted matrix factorizations
$(\varphi,\tau)$ and $(\varphi', \tau')$ to be
$(\varphi\oplus\varphi', \tau\oplus\tau')$. We call $(\varphi,\tau)$
\emph{reduced} if it is not isomorphic to a factorization having a
trivial direct summand.
\end{defn}

We denote by $TMF_A(f)$ the category whose objects are all twisted
left matrix factorizations of $f$ over $A$ and whose morphisms are
defined as above. As previously mentioned, $TMF_A(f)$ has a zero
object and all finite direct sums. Since morphisms are pairs of module
maps, monomorphisms and epimorphisms are determined
componentwise. Thus we have the following obvious fact.

\begin{prop}
$TMF_A(f)$ is an abelian category.
\end{prop}

Proposition \ref{morphisms} shows that forming the resolution $\mathbf{\Omega}(\varphi,\tau)$ defines a functor from the abelian category of twisted matrix factorizations of $f$ over $A$ to the abelian category of complexes of finitely generated graded projective $B$-modules. We examine categories of twisted matrix factorizations more closely beginning in Section \ref{section:zhang}.

Next we show that a normal, regular homogeneous element gives rise to twisted matrix factorizations in many cases of interest. However, see Example \ref{Heisenberg}. Let $B=A/(f)$.

\begin{cons}
\label{resolution}
Let $M$ be a finitely generated graded left $B$-module with $\text{pd}_A M=1$. Let $0\rightarrow F\xrightarrow{\varphi}G\rightarrow 0$ be a minimal graded resolution of $M$ by graded free left $A$-modules. We have the commutative diagram
$$
\xymatrix{
F^{tw}\ar@{->}[r]^{\varphi^{tw}}\ar@{->}[d]_{\lambda^F_f} & G^{tw}\ar@{->}[d]^{\lambda^G_f}\\
F\ar@{->}[r]_{\varphi} &G\\
}
$$

\begin{rmk}
The homomorphisms $\varphi^{tw}: F^{tw}\rightarrow G^{tw}$ and $\varphi:F\rightarrow G$ are identical on the underlying abelian groups. If we fix bases for $F$ and $G$, and keep the same bases for $F^{tw}$ and $G^{tw}$, the matrices of $\varphi$ and $\varphi^{tw}$ with respect to these bases are different. The matrix of $\varphi^{tw}$ is obtained by applying $\sigma^{-1}$ to each entry of the matrix of $\varphi$.
\end{rmk}

Since $M$ is a $B$-module, $f _AM=0$, hence $\im\ \lambda^G_f\subseteq\im\ \varphi$. Thus by graded projectivity there exists a lift $\tau:G^{tw}\rightarrow F$ such that $\lambda^G_f=\varphi\tau$. Note that since $f$ is regular, $\lambda^G_f$ is injective, so $\tau$ is injective.
$$
\xymatrix{
F^{tw}\ar@{->}[r]^{\varphi^{tw}}\ar@{->}[d]_{\lambda^F_f} & G^{tw}\ar@{->}[d]^{\lambda^G_f}\ar@{->}[dl]_{\tau}\\
F\ar@{->}[r]_{\varphi} &G\\
}
$$
Next observe that $\varphi\tau\varphi^{tw}-\varphi\lambda^F_f=0$ and since $\varphi$ is injective, $\tau\varphi^{tw}=\lambda^F_f$.

Iterating this process and applying $B\tensor_A -$ yields the complex $\mathbf{\Omega}(\varphi,\tau)$ of Proposition \ref{exactness}:
$$\cdots\rightarrow B\tensor_A F^{tw}\xrightarrow{1\tensor\varphi^{tw}} B\tensor_A G^{tw}\xrightarrow{1\tensor \tau} B\tensor_A F\xrightarrow{1\tensor \varphi} B\tensor_A G\rightarrow 0$$

\end{cons}

\begin{prop}
\label{minRes}
The complex $\mathbf{\Omega}(\varphi,\tau)$ is exact. It is a minimal graded free resolution of $_BM$ if and only if $M$ has no $B$-free direct summand. 
\end{prop}

\begin{proof} 
  Exactness follows from Proposition \ref{exactness}. Since
  $F\xrightarrow{\varphi}G$ is a minimal resolution of $_AM$, the
  complex $F^{tw^i}\xrightarrow{\varphi^{tw^i}}G^{tw^i}$ is a minimal
  resolution of $_AM^{tw^i}$ for all $i\ge 0$. Thus we have $\im\
  \varphi^{tw^i}\subseteq A_+G^{tw^i}$ for all $i\ge 0$. It follows
  that $\im(1\tensor\varphi^{tw^i})\subseteq B_+(B\tensor_A G^{tw^i})$
  for all $i\ge 0$. Thus it suffices to consider the maps
  $1\tensor\tau^{tw^i}$.

Now, $_BM$ has a $B$-free direct summand if and only if the twisted module $M^{\overline{tw}^i}$ does.

Let $\overline{F}=B\tensor_A F$. Since $\mathbf{\Omega}(\varphi,\tau)$ is exact, for each $i\ge 0$ we have
$$\im(1\tensor \tau^{tw^i})\isom {\rm coker}(1\tensor\varphi^{tw^i})\isom B\tensor_A M^{tw^i}\isom M^{\overline{tw}^i}$$

Since  $\im(1\tensor\tau^{tw^i})$ is contained in the radical $B_+\overline{F}^{tw^{i-1}}$ if and only if no basis for the free $B$-module $\overline{F}^{tw^i}$ intersects $\im(1\tensor\tau^{tw^i})$, the result follows.
\end{proof}

\begin{cor}
\label{minRes2}
Under the hypotheses and notation of Proposition \ref{minRes}, for any integer $n$, the complex $\mathbf{\Omega}(\varphi,\tau)^{\overline{tw}^n}$ is a graded free resolution of $M^{\overline{tw}^n}$
\end{cor}

We can also express minimality of the resolution $\mathbf{\Omega}(\varphi,\tau)$ in terms of the twisted matrix factoriaztion.

\begin{lemma}
  Let $A$ be a connected, $\N$-graded, locally finite-dimensional
  $k$-algebra and $f\in A_+$ a homogeneous normal, regular element. A
  twisted matrix factorization $(\varphi,\tau)$ of $f$ over $A$ is
  reduced if and only if $\mathbf{\Omega}(\varphi,\tau)$ is a minimal
  graded free resolution.
\end{lemma}

\begin{proof}
The complex $\mathbf{\Omega}(\varphi,\tau)$ is minimal if and only if ${\rm coker}(1\tensor \varphi^{tw^i})$ and ${\rm coker}(1\tensor\tau^{tw^i})$ have no $B$-free direct summands for all $i\ge 0$. Since the functor $(-)^{tw}$ preserves direct sums and free modules, the complex $\mathbf{\Omega}(\varphi,\tau)$ is minimal if and only if ${\rm coker}(1\tensor \varphi)={\rm coker}\ \varphi$ and ${\rm coker}(1\tensor\tau)={\rm coker}\ \tau$ have no $B$-free direct summands. The latter holds if and only if $(\varphi,\tau)$ is not isomorphic to a twisted factorization $(\varphi', \tau')$ where $\varphi'$ or $\tau'$ has $\lambda_f^A$ as a summand. 
\end{proof}

Next we turn to periodicity. 
Clearly the complex $\mathbf{\Omega}(\varphi,\tau)$ is periodic of
period at most $2n$ if $\sigma$ has finite order $n$. In practice, the
period is often less than $2n$ (see Section \ref{section:examples} for
some examples). Even when $|\sigma|=\infty$, the complex may be
periodic, as we show in the next proposition.

\begin{prop}
\label{periodicity}
The complex $\mathbf{\Omega}(\varphi,\tau)$ is periodic if and only if $M^{\bar\sigma^n}\isom M$ as $B$-modules for some $n>0$. In particular, if $M^{\bar\sigma}\isom M$, then $\mathbf{\Omega}(\varphi,\tau)$ has period at most 2.
\end{prop}

\begin{proof}
Suppose there exists an integer $n>0$ such that $M^{\bar\sigma^n}\isom M$. By Lemma \ref{stability} and Corollary \ref{minRes2}, $\mathbf{\Omega}(\varphi,\tau)$ is periodic of period at most $2n$.

Conversely, suppose there exists an integer $p>0$, an integer $N$, and a degree $-p$ morphism of complexes $\Phi:\mathbf{\Omega}(\varphi,\tau)(N)\rightarrow \mathbf{\Omega}(\varphi,\tau)$ such that $\Phi_{i+p}:\Omega_{i+p}(N)\rightarrow \Omega_i$ is an isomorphism for all $i\ge 0$. Since $\Phi^2$ also has this property, we may assume $p=2n$ is even. 

By construction, minimal generators of $\Omega_{i+p}$ can be taken to be minimal generators of $\Omega_i$ with degrees shifted up by $nd$. 
It follows that $N=nd$. Thus the following diagram commutes, and $$M^{\overline{\sigma}^n}\isom{\rm coker}(1\tensor\varphi^{\sigma^n})\isom {\rm coker}(1\tensor\varphi)\isom M$$
$$
\xymatrix{
B\tensor_A F^{\sigma^n}\ar@{->}[r]^{1\tensor \varphi^{\sigma^n}}\ar@{->}[d]_{\Phi} & B\tensor_A G^{\sigma^n}\ar@{->}[d]^{\Phi}\\
B\tensor_A F\ar@{->}[r]^{1\tensor\varphi} & B\tensor_A G\\
}
$$
\end{proof}

\section{Equivalent Categories of Twisted Matrix Factorizations}
\label{section:zhang}

In \cite{Zhang}, Zhang completely characterized pairs of graded $k$-algebras whose categories of graded modules are equivalent. With that characterization in mind, we consider the question of when categories of twisted matrix factorizations are equivalent. The following easy fact is useful later.

\begin{prop}
\label{rescale}\ 
\begin{enumerate}
\item For any scalar $\nu\in k^{\times}$, the categories $TMF_A(f)$ and $TMF_A(\nu f)$ are equivalent. 
\item Let $\phi:A\rightarrow A$ be a graded automorphism of $A$. Then $TMF_A(f)\approx TMF_A(\phi(f))$.
\end{enumerate}
\end{prop}

\begin{proof}
For (1), the functors $(\varphi,\tau)\mapsto (\varphi,\nu\tau)$ and $(\varphi,\tau)\mapsto (\varphi,\nu^{-1}\tau)$ are easily seen to be inverse equivalences. For (2), first observe that $\phi\sigma\phi^{-1}$ is the normalizing automorphism for $\phi(f)$. Applying the functor $(-)^{\phi^{-1}}$ to any twisted matrix factorization of $f$ over $A$ produces the desired equivalence.
\end{proof}

We briefly recall the basic definitions underlying Zhang's graded Morita equivalence and encourage the interested reader to see \cite{Zhang} for more details.

A (left) \emph{twisting system} for $A$ is a set $\zeta=\{\zeta_n\ |\ n\in\Z\}$ of graded $k$-linear automorphisms of $A$ such that $\zeta_n(\zeta_m(x)y)=\zeta_{m+n}(x)\zeta_{n}(y)$ for all $n,m,\ell\in\Z$ and $x\in A_{\ell}$, $y\in A_{m}$. 
For example, if $\phi$ is a graded $k$-linear automorphism of $A$, then setting $\zeta_n=\phi^n$ for all $n\in\Z$ gives a twisting system. 

Given a twisting system $\zeta$, the \emph{Zhang twist} of $A$ is the graded $k$-algebra $A^{\zeta}$ where $A^{\zeta}=A$ as graded $k$-vector spaces and for all $x\in A_{\ell}$ and $y\in A_{m}$, multiplication in $A^{\zeta}$ is given by $x\ast y=\zeta_m(x)y$. Likewise, if $M$ is a graded left $A$-module, the twisted left $A^{\zeta}$-module $M^{\zeta}$ has the same underlying graded vector space as $M$, and for $m\in M_n$ and $z\in A_{\ell}$, $z\ast m=\zeta_n(z)m$. Finally, we note that if $\varphi:M\rightarrow N$ is a degree 0 homomorphism of graded left $A$-modules, $\varphi:M^{\zeta}\rightarrow N^{\zeta}$ is also a degree 0 homomorphism of graded left $A^{\zeta}$-modules.

\begin{rmk}
\label{twistIsom}
Aside from the use of the letter $\zeta$, the notation for the twisted module is identical to that used for the functor $(-)^{\sigma}$ above. However, the notions are not the same. One important difference is that for an integer $n$, the free left $A$-modules $A(n)^{\sigma}$ and $A^{\sigma}(n)$ are identical, whereas the free left $A^{\zeta}$-modules $A(n)^{\zeta}$ and $A^{\zeta}(n)$ - which have the same underlying graded vector space - are generally not identical, but are isomorphic via the map $\zeta_{-n}$. In light of this subtlety, the following simple lemma is not entirely trivial.
\end{rmk}

\begin{lemma}
\label{naturality}
Let $A$ be a connected, $\N$-graded, locally finite-dimensional $k$-algebra. Let $f\in A_d$ be a normal, regular homogeneous element with normalizing automorphism $\sigma$. Let $\zeta=\{\zeta_n\ |\ n\in\Z\}$ be a left twisting system. 
\begin{enumerate}
\item If $\zeta_n(f)=c^nf$ for some $c\in k^{\times}$ and for all $n\in\Z$, then $f$ is normal and regular in $A^{\zeta}$ with normalizing automorphism $\hat{\sigma}(a)=c^{-\deg a}\sigma\zeta_d(a)$.
\item If $\zeta$ further satisfies $\zeta_n\sigma \zeta_d=\sigma\zeta_{n+d}$ for all $n\in\Z$ we have $(A^{tw})^{\zeta}\isom(A^{\zeta})^{\hat{tw}}:=(A^{\zeta})^{\hat{\sigma}}(-d)$ as free left $A^{\zeta}$-modules.
\end{enumerate}
\end{lemma}

If the twisting system $\zeta$ is ``algebraic,'' meaning $\zeta_{n}\zeta_{m}=\zeta_{n+m}$ for all $n,m\in\Z$, the additional hypothesis of (2) becomes $\sigma\zeta_n=\zeta_n\sigma$ for all $n\in\Z$. In the common case where $\zeta_n=\phi^n$ for a $k$-linear automorphism $\phi:A\rightarrow A$, one needs only that $\sigma\phi=\phi\sigma$.

\begin{proof}
Let $a\in A_{n}$ be an arbitrary homogeneous element. To prove (1), we have $$a\ast f = \zeta_d(a)f = f\sigma\left(\zeta_d(a)\right)=\zeta_n^{-1}(f)\ast \sigma\zeta_d(a)=c^{-n}f\ast\sigma\zeta_d(a)=f\ast \hat\sigma(a)$$  Thus $f$ is normal in $A^{\zeta}$. The equation also shows the regularity of $f$ in $A^{\zeta}$ follows from the regularity of $f$ in $A$, so $\hat\sigma$ is the normalizing automorphism.

For (2), first observe that $a\mapsto c^{\deg a}a$ defines a graded algebra automorphism $\lambda_{c}$ of $A^{\zeta}$. For any graded left $A^{\zeta}$-module $M$, $M\isom M^{\lambda_{c}}$ via the map $m\mapsto c^{\deg m}m$ which we also denote $\lambda_{c}$. 

Now, $(A^{tw})^{\zeta}$ and $(A^{\zeta})^{\hat{tw}}$ have the same underlying graded vector space as $A$. We compute the left $A^{\zeta}$ action  on both modules. With $a$ as above and $b\in A_m$, $A^{\zeta}$ acts on $(A^{\zeta})^{\hat{tw}}$ by
$$a\bullet b = \hat{\sigma}(a)\ast b = \zeta_m(\hat{\sigma}(a))b=\zeta_m c^{-\deg a}\sigma\zeta_d(a)b=c^{-\deg a}\zeta_m\sigma\zeta_d(a)b$$
and on $(A^{tw})^{\zeta}$ by
$$a\bullet b = \zeta_{m+d}(a)\cdot b= \sigma\zeta_{m+d}(a)b$$
since $b\in A_m=A^\sigma(-d)_{m+d}$. Thus $((A^{\zeta})^{\hat{tw}})^{\lambda_c}=(A^{tw})^{\zeta}$ and the result follows.
\end{proof}

For completeness, we mention the left module version of Zhang's theorem on graded Morita equivalence.

\begin{thm}[\cite{Zhang}]
Let $k$ be a field and let $A$ and $A'$ be connected graded $k$-algebras with $A_1\neq 0$. Then $A\isom A'^{\zeta}$ for some twisting system $\zeta$ if and only if the categories $A$-GrMod and $A'$-GrMod are equivalent. 
\end{thm}

The equivalence is given by $M\mapsto M^{\zeta}$ for any graded $A'$-module $M$ and is the identity on morphisms. We have the following.

\begin{thm}\label{twistEquiv}
Let $A$ be a connected, $\N$-graded locally finite dimensional $k$-algebra. Let $f\in A_d$ a normal, regular homogeneous element of degree $d$ with normalizing automorphism $\sigma$. Let $\zeta=\{\zeta_n\ |\ n\in\Z\}$ be a twisting system such that for all $n\in\Z$, $\zeta_n\sigma\zeta_d=\sigma\zeta_{n+d}$ and $\zeta_n(f)=c^nf$ for some $c\in k^{\times}$. Then the categories $TMF_A(f)$ and $TMF_{A^{\zeta}}(f)$ are equivalent.
\end{thm}

\begin{proof}
By Proposition \ref{rescale}, it suffices to prove that $TMF_A(f)$ is equivalent to  $TMF_{A^{\zeta}}(c^df)$.

Let $(\varphi:F\rightarrow G,\tau:G^{tw}\rightarrow F)$ be a twisted left matrix factorization of $f$ over $A$. Let $\lambda_c:(G^{\zeta})^{\hat{tw}}\rightarrow ((G^{\zeta})^{\hat{tw}})^{\lambda_c}$ be the graded isomorphism $m\mapsto c^{\deg m}m$ as in the proof of Lemma \ref{naturality}. By Lemma \ref{naturality}(2) and the note preceding Remark \ref{twistIsom}, $$(\varphi:F^{\zeta}\rightarrow G^{\zeta},\tau\lambda_c:(G^{\zeta})^{\hat{tw}}\rightarrow F^{\zeta})$$ is a twisted matrix factorization of $c^df$ over $A^{\zeta}$. The functoriality of Zhang's category equivalence $(-)^{\zeta}$ ensures any morphism $(\alpha,\beta):(\varphi,\tau)\rightarrow (\varphi',\tau')$ of twisted factorizations over $A$ remains a morphism over $A^{\zeta}$. This defines a functor $TMF_A(f)\rightarrow TMF_{A^{\zeta}}(f)$.

The inverse equivalence is given by applying the inverse twisting system $\zeta^{-1}=\{\zeta_n^{-1}\ |\ n\in\Z\}$ to a twisted matrix factorization over $A^{\zeta}$ and replacing $\lambda_c$ by $\lambda_{c^{-1}}$ in the above construction.
\end{proof}

\begin{cor}
The equivalence given in the preceding theorem completes a commutative diagram of functors
$$\xymatrix{
TMF_A(f) \ar[r] \ar[d]_{\coker} & TMF_{A^\zeta}(f) \ar[d]^{\coker} \\
A\text{-GrMod} \ar[r]^{Z} & A^{\zeta}\text{-GrMod}
}$$
where $Z$ denotes Zhang's equivalence of categories, and $\coker$ sends the
twisted matrix factorization $(\varphi,\tau)$ to $\coker \varphi$.
\end{cor}

We do not know an example of a twisting system $\zeta$ where $f$
remains normal and regular in $A^{\zeta}$ but $TMF_A(f)$ and
$TMF_{A^{\zeta}}(f)$ are inequivalent.

In some cases, a normal, regular element can become central in an
appropriate Zhang twist. By Proposition \ref{periodicity}, twisted
matrix factorizations of a central element produce resolutions with
period at most 2. Example \ref{ex:ore} below illustrates the following
important subtlety.

\begin{cor}
The period of a periodic minimal free resolution need not be invariant under a Zhang twist.
\end{cor}

\section{Noncommutative hypersurfaces}
\label{section:hypersurfaces}

The bijection between periodic minimal free resolutions and reduced matrix factorizations over local rings hinges on the Auslander-Buchsbaum formula. Before giving a noncommutative version of this correspondence, we recall J\o rgensen's version of Auslander-Buchsbaum for connected graded $k$-algebras. 

Throughout this section, let $A$ be a connected, $\N$-graded, locally
finite dimensional $k$-algebra and additionally assume $A$ is left
noetherian. Let $M$ be a finitely generated graded left
$A$-module. The \emph{depth} of $M$
is $$\depth_A(M)=\inf\{i\ |\ \Ext^i_A(k,M)\neq 0\}$$ Note that
$\depth_A(M)$ is either an integer or $\infty$. We have the following
special case of J\o rgensen's Auslander-Buchsbaum theorem.

\begin{thm}[\cite{Jorg}]
\label{AB}
With $A$ and $M$ as above, if the spaces $\Ext_A^i(k,A)$ are finite dimensional $k$-vector spaces for all $0\le i\le \depth_A(A)$ and if $\pd_A(M)<\infty$, then the Auslander-Buchsbaum formula
$$\pd_A(M)+\depth_A(M)=\depth_A(A)$$
holds for $M$. 
\end{thm}

Theorem \ref{AB} says the Auslander-Buchsbaum formula holds for all finitely generated graded left modules over left noetherian Artin-Schelter regular algebras. 

\begin{defn}
Let $A$ be a connected, $\N$-graded, locally finite-dimensional $k$-algebra. Then $A$ is \emph{Artin-Schelter regular} (resp.\ \emph{Artin-Schelter Gorenstein}) of dimension $d$ if
\begin{enumerate}
\item ${\rm gl.dim}(A)=d<\infty$ (resp.\ ${\rm inj.dim}(A)=d<\infty$ on both sides)
\item ${\rm GKdim}(A)=d$
\item $\Ext_A^i(k,A)=\delta_{i,d}k$
\end{enumerate}
\end{defn}

We frequently abbreviate these conditions AS-regular and
AS-Gorenstein. We also note that the results below in which $A$ is a
left noetherian AS-regular or AS-Gorenstein algebra do not require an
explicit assumption that the Gelfand-Kirillov dimension is finite.

The following fact is a consequence of the long exact sequence in
cohomology. We omit the straightforward proof.

\begin{lemma}
\label{SES}
Let $0\rightarrow M'\rightarrow M\rightarrow M''\rightarrow 0$ be a short exact sequence of graded left $A$ modules. If $\depth_A(M)>\depth_A(M'')$ then $\depth_A(M')=\depth_A(M'')+1$. 
\end{lemma}

Next, we need several facts relating the depth of $A$ to the depth of $A/(f)$ when $f$ is a normal, regular element. 
This result is classically known as Rees' Lemma; we include a proof here for completeness.

\begin{lemma}
\label{depthFacts}
Let $A$ be a connected $\N$-graded $k$-algebra. Let $f\in A$ be a normal, regular homogeneous element and let $B=A/(f)$. Then 
\begin{enumerate}
\item $\depth_B(B)=\depth_A(A)-1$
\item $\depth_A(M)=\depth_B(M)$ for any finitely generated, graded left $B$-module $M$. 
\item If $A$ is AS-Gorenstein, so is $B$.
\end{enumerate}
\end{lemma}

\begin{proof}
Let $N$ be a finitely generated, graded left $A$-module and consider the Cartan-Eilenberg change-of-rings spectral sequence
$$\Ext^p_B(k,\Ext^q_A(B,N))\Rightarrow \Ext^{p+q}_A(k,N)$$
Since $f\in A$ is regular, $\pd_A B=1$, hence the spectral sequence has only two nonzero rows. 
$$E_2^{p,0}=\Ext_B^p(k,\Hom_A(B,N))=\Ext_B^p(k,N^f)$$
$$E_2^{p,1}=\Ext_B^p(k,\Ext_A^1(B,N))=\Ext_B^p(k,N/fN)$$
where $N^f=\{n\in N\ |\ fn=0\}$.
The associated long exact sequence is (\cite{CE} Theorem XV.5.11)
$$\rightarrow E_2^{n,0}\rightarrow \Ext_A^n(k,N)\rightarrow E_2^{n-1,1}\rightarrow E_2^{n+1,0}\rightarrow \Ext_A^{n+1}(k,N)\rightarrow E_2^{n,1}\rightarrow$$
To prove the first statement, set $N=A$. Since $f$ is regular, $N^f=0$ and we have 
\begin{equation}
\Ext_A^n(k,A)\cong \Ext_B^{n-1}(k,B)
\end{equation}
The formula $\depth_B(B)=\depth_A(A)-1$ follows.

For the second statement, let $M$ be a finitely generated, graded left $B$-module. Let $N=M$ viewed as an $A$-module via the natural quotient map. Then $N^f=M$ and $fN=0$. For $i< d=\depth_B(M)$
$$ \Ext_B^{i}(k,M)\rightarrow \Ext_A^{i}(k,M)\rightarrow \Ext_B^{i-1}(k,M)$$
is exact, hence $\Ext_A^i(k,M)=0$, and 
$$\Ext_B^{d-2}(k,M)\rightarrow\Ext_B^{d}(k,M)\rightarrow \Ext_A^{d}(k,M)\rightarrow \Ext_B^{d-1}(k,M)$$
is exact, so $\Ext_A^{d}(k,M)\neq 0$. This shows $\depth_A(M)=\depth_B(M)$.

For the last statement, assume $A$ is AS-Gorenstein of dimension $\mu={\rm inj.dim}(A)$. Then ${\rm inj.dim(B)}=\mu-1$ and ${\rm GKdim}(B)={\rm GKdim}(A)-1$ (\cite{Lev} Theorem 3.6, Lemma 5.7).
Finally, since $\Ext_A^i(k,A)=\delta_{i,\mu}k$, equation (1) gives $\Ext_B^i(k,B)=\delta_{i,\mu-1}k$. Thus $B$ is AS-Gorenstein of dimension $\mu-1$.
\end{proof}

Theorem \ref{AB} and Lemma \ref{depthFacts} imply a graded $B$-module $M$ with $\pd_A(M)=1$ satisfies ${\rm depth}_B(M)={\rm depth}_B(B)$. By analogy with the commutative case, it is tempting to call such a module ``maximal Cohen-Macaulay.'' In his notes \cite{Buch}, Buchweitz defined the notion of a maximal Cohen-Macaulay module over any ring which is both left and right noetherian and has finite left and right injective dimension. We adopt a graded version of Buchweitz's definition for left noetherian rings. A finitely generated graded module $M$ over a connected, $\N$-graded, locally finite dimensional left noetherian $k$-algebra $B$ of finite left and right injective dimension is called \emph{maximal Cohen-Macaulay} if and only if $\Ext_B^i(M,B)=0$ for $i\neq 0$. The next Lemma shows that defining maximal Cohen-Macaulay modules in terms of depth is equivalent in our case.

\begin{lemma}
\label{MCM}
Let $A$ be a left noetherian, AS-regular algebra. Let $f\in A_+$ be a homogeneous normal, regular element and let $B=A/(f)$. Then for any finitely generated graded left $B$-module $M$, $\pd_A(M)=1$ if and only if $\Ext_B^i(M,B)=0$ for all $i\neq 0$.
\end{lemma}

\begin{proof}
We have $\pd_A(M)=1$ if and only if $\Ext_A^i(M,A)=0$ for all $i>1$. One direction of this is clear. The other is J\o rgensen's Ext-vanishing theorem \cite{Jorg2}. Let $d=\deg f$. Since $0\rightarrow A(-d)\xrightarrow{f}A\rightarrow B\rightarrow 0$ is a minimal graded free resolution of $_AB$, we see that $\Ext_A(B,A)$ is concentrated in homological degree 1 and $\Ext^1_A(B,A)\isom B(d)$ as graded left $B$-modules. Then the change of rings spectral sequence
$$\Ext_B^p(M,\Ext^q_A(B,A))\Rightarrow \Ext_A^{p+q}(M,A)$$
shows $\Ext_B^i(M,B)=0$ for $i\neq 0$ if and only if $\Ext_A^i(M,A)=0$ for all $i>1$.
\end{proof}

If $\cdots \rightarrow P_2\rightarrow P_1\rightarrow P_0\rightarrow M\rightarrow 0$ is an exact sequence of $B$-modules, we denote the $j$-th syzygy module $\im(P_j\rightarrow P_{j-1})$ by $\Omega_j(M)$ where $\Omega_0(M)=M$.

\begin{prop}
\label{syzygy}
Let $A$ be a left noetherian Artin-Schelter regular algebra of dimension $d$. Let $f\in A_+$ be a normal, regular homogeneous element and let $B=A/(f)$. Let $M$ be a finitely generated, graded left $B$-module and $\mathbf{P}$ a graded projective $B$-module resolution of $M$. Then $\pd_A(\Omega_i(M))=1$ for some $0\le i\le d$. 
\end{prop}

\begin{rmk*}
Recalling the construction of twisted matrix factorizations in Section
\ref{section:MF}, the Proposition shows their ubiquity over noetherian
AS-regular algebras in the presence of a normal regular element.
\end{rmk*}

\begin{proof} 
First, we show $\depth_B(B)\ge \depth_B(M)$. Indeed, by Theorem \ref{AB} and Lemma \ref{depthFacts} we have
\begin{align*}
\depth_B(M)=\depth_A(M)&=\depth_A(A)-\pd_A(M)\\
&=\depth_B(B)+1-\pd_A(M)
\end{align*}
Since $fM=0$, we have $\pd_A(M)>0$, hence $\depth_B(B)\ge \depth_B(M)$.
If $\depth_B(M)=\depth_B(B)$, the equation above shows $\pd_A(M)=1$, so assume $\depth_B(M)=i<\depth_B(B)$. Graded projective modules are graded free so $\depth_B(P_j)=\depth_B(B)$ for all $j\ge 0$. Since $$0\rightarrow \Omega_1(M)\rightarrow P_0\rightarrow M\rightarrow 0$$ is exact, $\depth_B(\Omega_1(M))=i+1$ by Lemma \ref{SES}. Inductively applying Lemma \ref{SES} to the exact sequence
$$0\rightarrow \Omega_{j+1}(M)\rightarrow P_j\rightarrow \Omega_{j}(M)\rightarrow 0$$
we obtain $\depth_B(\Omega_{d-i}(M))=\depth_B(B)$. It follows that $$\pd_A(\Omega_{d-i}(M))=1$$ as desired.
\end{proof}

We are ready to prove our main theorem. 
For ease of notation, we will no longer specify the degree of the
normal, regular homogeneous element $f$ and instead reserve $d$ for
the dimension of the ambient AS-regular algebra.  The functor
$(-)^{tw}$ continues to denote the composition of $(-)^{\sigma}$ with
an appropriate degree shift.

\begin{thm}
\label{mainThm}
Let $A$ be a left noetherian Artin-Schelter regular algebra of
dimension $d$. Let $f\in A_+$ be a homogeneous normal regular element
and let $\sigma$ be its normalizing automorphism. Let $B=A/(f)$. If
$$\mathbf{Q}: \cdots \rightarrow Q_2 \rightarrow Q_1\rightarrow Q_0$$
is a minimal graded free left $B$-module resolution of a finitely generated graded left $B$-module $M$, then 
\begin{enumerate} 
\item The truncated complex $\cdots \rightarrow Q_{d+2}\rightarrow
  Q_{d+1}$ is chain isomorphic to $\mathbf{\Omega}(\varphi,\tau)$ for
  some reduced twisted left matrix factorization $(\varphi,\tau)$.
\end{enumerate}
Assuming further that $|\sigma|<\infty$, we have
\begin{enumerate}
\item[(2)] $\mathbf{Q}$ becomes periodic of period at most $2|\sigma|$ after $d+1$ steps. 
\item[(3)] $\mathbf{Q}$ is periodic (of period at most $2|\sigma|$) if
  and only if $\pd_A(M)=1$ and $M$ has no graded free $B$-module summand.
\item[(4)] Every periodic minimal graded free left module resolution over $B$
  has the form $\mathbf{\Omega}(\varphi,\tau)$ for some reduced
  twisted left matrix factorization $(\varphi,\tau)$ of $f$ over $A$.
\end{enumerate}
\end{thm}

\begin{proof}

By Proposition \ref{syzygy}, we have $\pd_A(\Omega_i(M))=1$ for some $0\le i\le d$. If $\Omega_i(M)=\Omega_i'(M)\oplus F$ where $F$ is a graded free $B$-module and $\Omega_i'(M)$ has no free summand, then $\pd_A(\Omega_i'(M))=1$. By Construction \ref{resolution} and Proposition \ref{minRes} there exists a twisted left matrix factorization $(\varphi,\tau)$ such that  $\mathbf{\Omega}(\varphi,\tau)$ is a periodic minimal graded free resolution of $\Omega_i'(M)$. If $F[i]$ denotes the free module $F$ viewed as a complex concentrated in homological degree $i$, it follows that 
$$\widetilde{\mathbf{Q}}:\mathbf{\Omega}(\varphi,\tau) \oplus F[i]\rightarrow Q_{i-1}\rightarrow\cdots\rightarrow Q_0$$
is a minimal graded free resolution of $M$ (or, if $i=0$, $\mathbf{\Omega}(\varphi,\tau) \oplus F$ is a resolution). By uniqueness of minimal resolutions, $\widetilde{\mathbf{Q}}\isom \mathbf{Q}$. Truncating each complex at homological degree $i+1$ and recalling that if $(\varphi,\tau)$ is a twisted matrix factorization, so are $(\varphi^{tw},\tau^{tw})$ and $(\tau,\varphi^{tw})$, we have established (1).

If $|\sigma|<\infty$, the resolution $\widetilde{\mathbf{Q}}$ is periodic of period at most $2|\sigma|$ after $i+2$ steps and 
 $\rank\ \widetilde{Q}_j=\rank\ \widetilde{Q}_{i+2}$ for all $j\ge i+2$. This proves (2). Setting $i=0$ and $\Omega_i'(M)=\Omega_i(M)$, we also obtain the ``if'' direction of $(3)$.
 
Now suppose that $\mathbf{Q}$ is periodic of period $p$. If $\Omega_i(M)$ has a free summand, $\rank\ Q_{p+i+1}=\rank\ Q_{i+1}>\rank\ Q_{i+2}$. But this is impossible, since $\mathbf{Q}$ and $\widetilde{\mathbf{Q}}$ are isomorphic minimal free resolutions. Thus $\Omega_i(M)$ has no free direct summand and $\widetilde{\mathbf{Q}}:\mathbf{\Omega}(\varphi,\tau)\rightarrow Q_i\rightarrow \cdots\rightarrow Q_0$ is a minimal free resolution of $M$. By Lemma \ref{resFacts}, $\rank\ \widetilde{Q}_j=\rank\ \widetilde{Q}_0$ for all $j\ge 0$, so $M$ has no free direct summand.
 
By graded periodicity, $M=\text{coker}(Q_1\rightarrow Q_0)$ is
isomorphic to $\text{coker}(1\tensor\varphi^{\sigma^m})$ or
$\text{coker}(1\tensor\tau^{\sigma^m})$ for some $m$. Since both maps
lift to injective maps of free $A$-modules, $\pd_A(M)=1$. This
completes the proof of (3). Since $(\varphi^{tw^m},\tau^{tw^m})$ and
$(\tau^{tw^m},\varphi^{tw^{m+1}})$ are also twisted matrix
factorizations, (4) follows as well.
\end{proof}

Taking $A$ to be the polynomial ring $k[x_1,\ldots, x_n]$, we recover a graded version of Theorem 6.1 of \cite{Eis} as a special case of Theorem 4.5. We remark that the analogous theorem in \cite{Eis} relies on the existence of regular sequences of length depth($A$), whereas our proof necessarily avoids this assumption. 

As a first corollary, we have the following useful fact.

\begin{cor}
\label{freeSummands}
Let $A$, $f$, and $B$ as in the theorem. Assume $|\sigma|<\infty$.  If
$(\mathbf{Q},\partial)$ is a minimal graded free left $B$-module
resolution of a finitely generated module, then $\im\ \partial_k$ has
no free summands for $k\ge d+1$.
\end{cor}

We also see that resolutions of the trivial module $_Bk$ have a very rigid structure. Note we do not need to assume $|\sigma|<\infty$.

\begin{cor}
  Let $A$, $f$, and $B$ as in the theorem. There exists a minimal graded free
  resolution of the trivial $B$-module $_Bk$ which becomes periodic of
  period at most 2 after $d+1$ steps.
\end{cor}

\begin{proof}
Let $\mathbf{Q}$ be a minimal graded free resolution of $_Bk$. By Theorem \ref{mainThm}, 
there exists a twisted left matrix factorization $(\varphi,\tau)$ of $f$ such that
$\mathbf{\Omega}(\varphi,\tau)$ is a minimal graded projective resolution of the $(d+1)$-st syzygy $\Omega_{d+1}(_Bk)$.

Let $\bar{\sigma}:B\rightarrow B$ be the automorphism induced by $\sigma$. Clearly, ${_Bk^{\bar{\sigma}}}\isom { _Bk}$, so there exists a chain isomorphism $\Phi:\mathbf{Q}^{\bar{\sigma}}\rightarrow \mathbf{Q}$ by Lemma \ref{stability}. By the 5-Lemma, $\Phi_{d+1}$ restricts to a graded $B$-module isomorphism $$(\Omega_{d+1}(_Bk))^{\bar{\sigma}}\isom \Omega_{d+1}(_Bk)$$ The result follows from Proposition \ref{periodicity}.
\end{proof}

Recall that a graded free resolution $(\mathbf{P}_{\bullet}, d_{\bullet})$ is called \emph{linear} if $P_i$ is generated in degree $i$ for all $i\ge 0$. 

\begin{cor}
Let $A$, $f$, and $B$ be as in the theorem. Additionally assume $f$ is quadratic. Then $\mathbf{Q}$ becomes a linear free resolution after $d+1$ steps. 
\end{cor}

\begin{cor}
  There exist bijections between isomorphism classes of reduced
  twisted left matrix factorizations of $f$ over $A$, isomorphism classes
  of nontrivial periodic minimal graded free resolutions of finitely
  generated graded left $B$-modules, and isomorphism classes of maximal
  Cohen-Macaulay left $B$-modules without free summands.
\end{cor}

\section{Homotopy category of twisted matrix factorizations}
\label{section:categories}

In \cite{Buch}, Buchweitz established an equivalence between the
stable category of maximal Cohen-Macaulay modules over a noetherian
ring $B$ with finite left and right injective dimensions and a
quotient of the bounded derived category of modules over $B$ now
called the singularity category. He noted the equivalence also holds
in the graded case. In \cite{Orlov1}, Orlov proved that if $B$ is a
graded factor algebra of a finitely generated, connected, $\N$-graded
noetherian $k$-algebra $A$ of finite global dimension by a central,
regular element $W$, then the stable bounded derived category of
graded $B$-modules is equivalent to a category Orlov called ``the category
of graded $D$-branes of type $B$ for the pair $(B,W)$.''
In this section we extend Orlov's result to factors of
left noetherian AS-regular algebras by regular, normal elements.  Much
of Orlov's work goes through with the obvious necessary changes. The
key difference is that $|\sigma|$ need not be finite in our case, so
we cannot appeal to periodicity of a resolution.

We continue to let $A$ be a connected, noetherian, $\N$-graded,
locally finite dimensional $k$-algebra and $f\in A_+$ a normal,
regular homogeneous element with normalizing automorphism
$\sigma$. Let $B=A/(f)$.  We will also continue to consider left modules over $B$.

There is a natural functor of abelian categories $\mathcal
C:TMF_A(f)\rightarrow B$-$\GrMod$ given on objects by
$(\varphi,\tau)\mapsto {\rm coker}\ \varphi$. (Recall from the proof
of Proposition \ref{exactness} that ${\rm coker}\ \varphi$ is a
$B$-module.) A morphism $\Psi:(\varphi,\tau)\rightarrow
(\varphi',\tau')$ induces a well-defined map $\psi:{\rm coker}\
\varphi\rightarrow{\rm coker}\ \varphi'$ by $\pi'\Psi_G\pi^{-1}$ where
$\pi:G\rightarrow {\rm coker}\ \varphi$ and $\pi':G'\rightarrow {\rm
  coker}\ \varphi'$ are the canonical projections, and $\pi^{-1}$ is any
section of $\pi$.

The functor $\mathcal C$ is not essentially surjective\footnote{A
  functor $F : C \to D$ is \emph{essentially surjective} if every
  object in $D$ is isomorphic to $F(c)$ for some object $c$ in $C$.};
objects in the image of $\mathcal C$ are finitely generated
$B$-modules $M$ such that $\pd_A(M)=1$. By Lemma \ref{MCM}, if $A$ is
left noetherian and AS-regular, the image of $\mathcal C$ consists of
maximal Cohen-Macaulay modules.

We denote the full subcategory of maximal Cohen-Macaulay modules in
$B$-$\GrMod$ by $MCM(B)$. Following \cite{Buch}, we define the
category of \emph{stable} maximal Cohen-Macaulay modules, which we
denote $\underline{MCM}(B)$, to have the same objects as $MCM(B)$, but
for $M,N\in\underline{MCM}(B)$,
$$\Hom_{\underline{MCM}(B)}(M,N)=\Hom_B(M,N)/R $$
where $R$ is the subspace of morphisms which factor through a graded projective $B$-module. 

As in \cite{Buch} and \cite{Orlov1}, let $D^b(B)$ be the bounded
derived category of finitely generated graded left $B$-modules. A
complex in $D^b(B)$ is called \emph{perfect} if it is isomorphic in
$D^b(B)$ to a complex of finitely generated graded projective
modules. Perfect complexes form a full, triangulated subcategory
$D^b_{perf}(B)$ of $D^b(B)$. The \emph{singularity category of} $B$ is
defined to be the quotient category
$D^b_{sg}(B)=D^b(B)/D^b_{perf}(B)$.

As noted in \cite{Buch}, the composition $MCM(B)\rightarrow
D^b(B)\rightarrow D^b_{sg}(B)$, where the first functor takes a module
to its trivial complex, factors uniquely through the quotient
$MCM(B)\rightarrow \underline{MCM}(B)$, yielding a functor $\mathcal
G:\underline{MCM}(B)\rightarrow D^b_{sg}(B)$. Buchweitz showed that
$\mathcal G$ is an exact equivalence. The equivalence $\mathcal G$
induces a triangulated structure on $\underline{MCM}(B)$. (It is
possible to describe the triangulated structure independently, see
\cite{Buch}, but as we will not need it, we omit any details.)

Thus it is natural to consider a ``stable'' version of the category
$TMF_A(f)$. To motivate the definition, we remark that the category
$TMF_A(f)$ is equivalent to the category of doubly-infinite sequences
of graded free $A$-module homomorphisms of the form
$$\cdots\rightarrow F^{tw}\xrightarrow{\varphi^{tw}} G^{tw}\xrightarrow{\tau} F\xrightarrow{\varphi} G\xrightarrow{\tau^{tw^{-1}}} F^{tw^{-1}}\rightarrow \cdots$$
whose compositions are multiplication by $f$, and whose
morphisms are maps of sequences
$$\Psi=(\ldots,\Psi_F^{tw},\Psi_{G}^{tw}, \Psi_F, \Psi_G,\Psi_F^{tw^{-1}},\ldots)$$
satisfying the necessary commutation relations.  We adopt the
structure of the homotopy category of such sequences of graded
projective $A$-modules.

\begin{defn}
  A morphism $\Psi:(\varphi,\tau)\rightarrow (\varphi',\tau')$ is
  \emph{null homotopic} if there exists a pair $(s,t)$ of degree 0
  module homomorphisms $s:G\rightarrow F'$ and $t:F\rightarrow
  G'^{tw}$ such that $\Psi_G^{tw}=\varphi'^{tw}s^{tw}+t\tau$ and
  $\Psi_F=\tau't+s\varphi$.
\end{defn}

We denote by $hTMF_A(f)$ the quotient (homotopy) category of
$TMF_A(f)$ with the same objects, and whose morphisms are equivalence
classes of morphisms in $TMF_A(f)$ modulo null homotopic
morphisms. Observe that taking $s:A\rightarrow A$ to be the identity
map and $t:A\rightarrow A^{tw}$ to be zero shows the identity map
$({\rm id}_A,\lambda_f^A)\rightarrow ({\rm id}_A,\lambda_f^A)$ is null
homotopic. Thus $({\rm id}_A,\lambda_f^A)\isom 0$ in $hTMF_A(f)$
\footnote{Recall that objects $(\varphi,\tau)$,$(\varphi',\tau')$ are isomorphic
in $hTMF_A(f)$ if and only if there exist maps $\Phi$ and $\Psi$ between them such that
${\rm id}_{(\varphi,\tau)} - \Phi\Psi$ and ${\rm id}_{(\varphi',\tau')} - \Psi\Phi$ are null homotopic.}.
A similar calculation shows $(\lambda_f^A,{\rm id}_{A^{tw}})\isom 0$.
More generally we have the following.

\begin{lemma}
\label{trivial}
If $(\varphi,\tau)\in TMF_A(f)$ such that ${\rm coker}\ \varphi$ is a graded free $B$-module, then $(\varphi,\tau)\isom 0$ in $hTMF_A(f)$.
\end{lemma}

\begin{proof}
  The lemma is trivial if $\varphi=0$, so suppose $\varphi\neq 0$ and
  $M={\rm coker}\ \varphi={\rm coker}\
  (B\tensor_AF\xrightarrow{1\tensor\varphi} B\tensor_A G)$ is a graded
  free left $B$-module. Let $\psi:M\rightarrow B\tensor_A G$ be a
  graded splitting of the canonical projection $\pi$, viewed as a map
  of graded left $A$-modules. Since $B\tensor_A G$ is isomorphic (as a
  left $A$-module) to the cokernel of $\lambda_f^G:G^{tw}\rightarrow
  G$, lifting $\psi$ gives a commutative diagram with exact rows
$$
\xymatrix{
0\ar@{->}[r] & F\ar@{->}[r]^{\varphi}\ar@{->}[d]_{\Psi_F} & G\ar@{->}[d]^{\Psi_G}\ar@{->}[r] &M\ar@{->}[d]^{\psi}\ar@{->}[r]&0\\
0\ar@{->}[r] & G^{tw}\ar@{->}[r]^{\lambda_f^G}\ar@{->}[d]_{\tau} & G\ar@{->}[r]\ar@{->}[d]^{\rm id} & B\tensor_A G\ar@{->}[d]^{\pi}\ar@{->}[r]&0\\
0\ar@{->}[r] & F\ar@{->}[r]^{\varphi} & G\ar@{->}[r] &M\ar@{->}[r]&0\\
}
$$
Since $\pi\psi={\rm id}_M$, there exists $s:G\rightarrow F$ such that ${\rm id}_F-\tau\Psi_F=s\varphi$
and ${\rm id}_G-\Psi_G=\varphi s$ by the comparison theorem. (The
only $A$-module homomorphism $M\rightarrow G$ is the zero map.)
If we now set $t = \Psi_F$, the morphism $(\Psi_G,\tau\Psi_F)$ of twisted
matrix factorizations is chain homotopic to the identity map on
$(\varphi,\tau)$ via the pair $(s,t)$.  Indeed, it is clear that ${\rm id}_F = \tau t + s\varphi$. To see that
${\rm id}_G^{tw}=\varphi^{tw}s^{tw}+t\tau$, it suffices to show
$\Psi_G^{tw}=t\tau$. This follows from the equalities
$$0=\Psi_G^{tw}\varphi^{tw} - \lambda_f^{G^{tw}}\Psi_F^{tw} = \Psi_G^{tw}\varphi^{tw} - \Psi_F\lambda_f^F = \Psi_G^{tw}\varphi^{tw} - \Psi_F\tau\varphi^{tw}$$
and the injectivity of $\varphi$.
\end{proof}

\begin{rmk} \label{rmk:full}
  Implicit in the previous proof is the fact that $TMF_A(f)\rightarrow
  MCM(B)$, and hence the composite $\underline{\mathcal
    C}:TMF_A(f)\rightarrow \underline{MCM}(B)$, is a full functor.
\end{rmk}

Our next objective is to establish the following fact.

\begin{prop}
The category $hTMF_A(f)$ is a triangulated category.
\end{prop}

We begin with a few definitions.  The translation functor on $hTMF_A(f)$ is
given by $(\varphi,\tau)[1]=(-\tau^{tw^{-1}},-\varphi)$ on objects and by
$\Psi[1]=(\Psi_F^{tw^{-1}},\Psi_G)$ on morphisms. For any morphism
$\Psi:(\varphi,\tau)\rightarrow (\varphi',\tau')$ the \emph{mapping cone} of
$\Psi$ is the pair
$$C(\Psi)=(\gamma:F'\oplus G\rightarrow G'\oplus F^{tw^{-1}},\quad \delta:G'^{tw}\oplus F\rightarrow F'\oplus G)$$ where
$$\gamma=\begin{pmatrix}\varphi' & 0\\ \Psi_G & -\tau^{tw^{-1}}\\\end{pmatrix}\qquad\text{and}\qquad \delta=\begin{pmatrix}\tau' & 0\\ \Psi_F & -\varphi\\\end{pmatrix}.$$
By the above matrix notation, we mean that the maps $\gamma$ and $\delta$ are given as follows on ordered pairs:
$$\gamma(x',y) = \left(\varphi'(x') + \Psi_G(y), -\tau^{tw^{-1}}(y)\right) \qquad \delta(y',x) = \left(\tau'(y') + \Psi_F(x),-\varphi(x)\right).$$
It is straightforward to check that this pair is a twisted matrix
factorization and there exist canonical inclusion and projection
morphisms $i:(\varphi',\tau')\rightarrow C(\Psi)$ and
$p:C(\Psi)\rightarrow (\varphi,\tau)[1]$. Moreover, given a
commutative square of twisted factorizations
$$
\xymatrix{
(\varphi,\tau)\ar@{->}[r]^{\Psi}\ar@{->}[d]_{\Pi} & (\varphi',\tau')\ar@{->}[d]^{\Pi'}\\
(\gamma,\delta)\ar@{->}[r]_{\Phi} & (\gamma',\delta')\\
}
$$
an easy diagram chase shows $(\Pi'_G\oplus \Pi_F^{tw^{-1}},
\Pi'_F\oplus \Pi_G)$ defines a morphism $C(\Psi)\rightarrow C(\Phi)$.
Note that the complex $\mathbf{\Omega}(C(\Psi))$ is the mapping cone
of the induced morphism of complexes
$\mathbf{\Omega}(\Psi):\mathbf{\Omega}(\varphi,\tau)\rightarrow\mathbf{\Omega}(\varphi',\tau')$.

We define a \emph{standard triangle} to be any sequence of maps in $hTMF_A(f)$
$$(\varphi,\tau)\xrightarrow{\Psi} (\varphi',\tau')\xrightarrow{i} C(\Psi)\xrightarrow{p} (\varphi,\tau)[1].$$

We define a \emph{distinguished triangle} to be any triangle 
$$(\varphi,\tau)\xrightarrow{\Psi} (\varphi',\tau')\xrightarrow{\Psi'} (\varphi'',\tau'')\xrightarrow{\Psi''} (\varphi,\tau)[1]$$
isomorphic to a standard triangle. For any twisted factorization $(\varphi,\tau)$, the triangle 
$$(\varphi,\tau)\xrightarrow{{\rm id}}(\varphi,\tau)\rightarrow 0\rightarrow (\varphi,\tau)[1]$$ is distinguished.
To see this, consider the diagram
$$
\xymatrix{
(\varphi,\tau)\ar@{->}[r]^{{\rm id}}\ar@{->}[d]^{\rm id} &(\varphi,\tau)\ar@{->}[r]\ar@{->}[d]^{\rm id} & 0\ar@{->}[r]\ar@{->}[d]&
(\varphi,\tau)[1]\ar@{->}[d]^{\rm id} \\
(\varphi,\tau)\ar@{->}[r]^{{\rm id}} &(\varphi,\tau)\ar@{->}[r]^{i} &
C({\rm id})\ar@{->}[r]^{p} & (\varphi,\tau)[1] \\
}
$$
and note that ${\rm id}_{C({\rm id})}$ is null homotopic via the pair
$$s:G\oplus F^{tw^{-1}}\rightarrow F\oplus G\qquad t:F\oplus G\rightarrow G^{tw}\oplus F$$
both given by $(x,y)\mapsto (0,x)$.  Precomposing this homotopy with
the canonical inclusion $(\varphi,\tau) \rightarrow C({\rm id})$ shows
that $i$ is null homotopic.  Thus the diagram commutes in $hTMF_A(f)$,
and is hence an isomorphism of triangles in $hTMF_A(f)$.

To show $hTMF_A(f)$ is triangulated, it remains to show distinguished
triangles are closed under rotations and that the octahedral axiom
holds. The argument very closely follows the proof of Theorem IV.1.9
in \cite{GelMan}. We discuss only rotations of distinguished triangles
in detail, leaving the translation of the remainder of the proof from
\cite{GelMan} to the interested reader.

To verify the class of distinguished triangles is closed under rotations, it suffices to consider standard triangles. 

Let
$$(\varphi,\tau)\xrightarrow{\Psi} (\varphi',\tau')\xrightarrow{i} C(\Psi)\xrightarrow{p} (\varphi,\tau)[1]$$
be a standard triangle. To see the rotated triangle
$$(\varphi',\tau')\xrightarrow{i} C(\Psi)\xrightarrow{p} (\varphi,\tau)[1]\xrightarrow{-\Psi[1]} (\varphi',\tau')[1]$$
is distinguished, first observe that $C(i)$ is given by the pair
$$ \left( F'\oplus G \right) \oplus G'\xrightarrow{
\begin{pmatrix} \varphi' & 0 & 0\\ \Psi_G & -\tau^{tw^{-1}} & 0\\ {\rm id} & 0 & -(\tau')^{tw^{-1}}\\ \end{pmatrix}
} \left( G'\oplus F^{tw^{-1}} \right) \oplus (F')^{tw^{-1}}$$
$$ \left( G'^{tw}\oplus F \right) \oplus F'\xrightarrow{
\begin{pmatrix} \tau' & 0 & 0\\ \Psi_F & -\varphi & 0\\ {\rm id} & 0 & -\varphi'\\ \end{pmatrix}
} \left( F'\oplus G \right) \oplus G'.$$
Let $\Theta:(\varphi,\tau)[1]\rightarrow C(i)$ be the morphism defined by the pair
$$\Theta_{F^{tw^{-1}}} : F^{tw^{-1}} \xrightarrow{(0\ {\rm id}\ -\Psi_F^{tw^{-1}})} G'\oplus F^{tw^{-1}}\oplus (F')^{tw^{-1}}$$
$$\Theta_G:G\xrightarrow{(0\ {\rm id}\ -\Psi_G)} F'\oplus G\oplus G'$$
This gives a diagram
$$
\xymatrix{
(\varphi',\tau')\ar@{->}[r]^{i}\ar@{->}[d]^{\rm id} &C(\Psi)\ar@{->}[r]^{p}\ar@{->}[d]^{\rm id} &
(\varphi,\tau)[1]\ar@{->}[d]^{\Theta}\ar@{->}[r]^{-\Psi[1]} & (\varphi',\tau')[1]\ar@{->}[d]^{\rm id} \\
(\varphi',\tau')\ar@{->}[r]^{i} &C(\Psi)\ar@{->}[r]^{j} &
C(i)\ar@{->}[r]^{q} & (\varphi',\tau')[1] \\
}
$$
where $j$ and $q$ are the canonical morphisms for the mapping cone
$C(i)$. The first and last squares are easily seen to commute.

The middle square, however, commutes only up to homotopy.  The morphism $j-\Theta p$
is seen to be null homotopic via the pair of maps
$$s:G'\oplus F^{tw^{-1}}\rightarrow F'\oplus G\oplus G'\qquad t:F'\oplus G\rightarrow (G')^{tw}\oplus F\oplus F'$$
both given by $(x,y)\mapsto (0,0,x)$. To see that $\Theta$ is an
isomorphism in $hTMF_A(f)$, let $\pi:C(i)\rightarrow
(\varphi,\tau)[1]$ be the canonical projection. Then $\pi\Theta$ is
the identity on $(\varphi,\tau)[1]$ and ${\rm id}_{C(i)}-\Theta\pi$ is
seen to be null homotopic by precomposing the pair $(s,t)$ above with
the projection $C(i)\rightarrow C(\Psi)$. This shows the class of
distinguished triangles is closed under rotations.

\begin{thm}
\label{catEquiv}
Let $A$ be a left noetherian AS-regular algebra, $f \in A_+$ a homogeneous normal regular element,
and $B = A/(f)$.  Then the categories $hTMF_A(f)$, $\underline{MCM}(B)$, and $D^b_{sg}(B)$ are equivalent.
\end{thm}

\begin{proof} Since $\mathcal G$ is known to be an exact equivalence, it suffices to show $hTMF_A(f)\approx \underline{MCM}(B)$. 

The functor $\underline{\mathcal C}:TMF_A(f)\rightarrow
\underline{MCM}(B)$ factors through the projection to 
$hTMF_A(f)$ to complete the commutative diagram of functors
$$
\xymatrix{
TMF_A(f)\ar@{->}[r]^{\mathcal C}\ar@{->}[d]\ar@{->}[dr]^{\underline{\mathcal C}} &MCM(B)\ar@{->}[r]\ar@{->}[d] & D^b(B)\ar@{->}[d]\\
hTMF_A(f)\ar@{->}[r]^{\mathcal F}&\underline{MCM}(B)\ar@{->}[r]^{\mathcal G} & D^b_{sg}(B).\\
}
$$

To see this, it is enough to show that any null homotopic morphism
$\Psi:(\varphi,\tau)\rightarrow (\varphi',\tau')$ induces the zero map
in $\underline{MCM}(B)$. Specifically, we show the induced map
$\psi:{\rm coker}\ \varphi\rightarrow {\rm coker}\ \varphi'$ factors
through the graded projective module $B\tensor_A G'$. The morphism
$\Psi$ factors
as $$(\varphi,\tau)\xrightarrow{\Phi}(\gamma,\delta)\xrightarrow{\Pi}(\varphi',\tau')$$
through the twisted ``horseshoe'' factorization $(\gamma,\delta)$
where
$$\gamma:(G')^{tw}\oplus F'\xrightarrow{\begin{pmatrix}-\tau' & 0\\ {\rm id} & \varphi'\\ \end{pmatrix}} F'\oplus G',$$
$$\delta:(F')^{tw}\oplus (G')^{tw}\xrightarrow{\begin{pmatrix}-(\varphi')^{tw} & 0\\ {\rm id} & \tau'\\ \end{pmatrix}} (G')^{tw}\oplus F',$$
$\Phi_G=(s,\Psi_G)$, $\Phi_F=(t,\Psi_F)$ and $\Pi$ is the canonical projection onto the second factor. 

We claim ${\rm coker}\ \gamma = B\tensor_A G'$. For any $x\in F'$ and
$y\in G'$, $(x,y)=(0,y-\varphi'(x))$ in ${\rm coker}\ \gamma$. Thus
there is a surjection $G'\onto {\rm coker}\ \gamma$. The kernel of
this surjection consists of $z\in G'$ such that
$z=\varphi'\tau'(w)=fw$ for some $w\in (G')^{tw}$. So ${\rm coker}\
\gamma=G'/fG'=B\tensor_A G'$.

Now the induced maps ${\rm coker}\ \varphi\xrightarrow{\phi} B\tensor_A G'\xrightarrow{\pi} {\rm coker}\ \varphi'$ show the map $\psi$ induced by $\Psi$ factors through a graded projective module, hence is the zero map in $\underline{MCM}(B)$. Thus the functor $\mathcal F$ is well-defined. 

The triangulated structure on $\underline{MCM}(B)$ is induced by
$\mathcal G$, so to prove $\mathcal F$ is an exact functor, it
suffices to check that $\mathcal GF$ is exact. By Proposition
\ref{minRes}, $\mathbf{\Omega}((\varphi,\tau)[1])$ is exact. Thus
$$0\rightarrow{\rm coker}\ (-\varphi)\rightarrow B\tensor_A F^{tw^{-1}}\rightarrow {\rm coker}\ (-\tau^{tw^{-1}})\rightarrow 0$$
is a short exact sequence in $B$-\GrMod, and hence
$${\rm coker}\ (-\varphi)\rightarrow B\tensor_A F^{tw^{-1}}\rightarrow {\rm coker}\ (-\tau^{tw^{-1}})\rightarrow ({\rm coker}\ (-\varphi))[1]$$
is a distinguished triangle in $D^b_{sg}(B)$. Since $B\tensor_A
F^{tw^{-1}}$ is graded free, the first two morphisms are
zero. Rotating the triangle yields
$${\rm coker (-\tau^{tw^{-1}}})\isom({\rm coker}\ (-\varphi))[1] \isom ({\rm coker}\ \varphi)[1] =\mathcal (F(\varphi,\tau))[1]$$ in
$D^b_{sg}(B)$. Thus we have a natural isomorphism $(\mathcal F(\varphi,\tau))[1]\isom \mathcal F((\varphi,\tau)[1])$. That $\mathcal
F$ takes a standard triangle in $hTMF_A(f)$ to a distinguished triangle in
$D^b_{sg}(B)$ follows from this natural isomorphism, the fact that for
a morphism $\Psi:(\varphi,\tau)\rightarrow (\varphi',\tau')$,
$\mathbf{\Omega}(C(\Psi))$ is the mapping cone of
$\mathbf{\Omega}(\Psi):\mathbf{\Omega}(\varphi,\tau)\rightarrow\mathbf{\Omega}(\varphi',\tau')$,
and the usual property of mapping cones fitting into long exact
sequences in homology.

By Construction \ref{resolution}, $\underline{\mathcal C}$ is surjective on objects of $\underline{MCM}(B)$, hence the same is true of $\mathcal F$. 
Since $\underline{\mathcal C}$ is full by Remark \ref{rmk:full}, $\mathcal F$ is as well. 

To see that $\mathcal F$ is injective on objects, we show $\mathcal{GF}$ is. Suppose $\mathcal{GF}(\varphi,\tau)\isom 0$ in $D^b_{sg}(B)$. Then $M={\rm coker}\ \varphi$ admits a finite length graded free $B$-module resolution, so $\Ext_B^i(M,N)=0$ for all $N$ and all $i\gg 0$. By Proposition \ref{minRes}, $\mathbf{\Omega}(\varphi,\tau)$ is a graded free $B$-module resolution. Thus for some $n$, $\Ext_B^i({\rm coker}(1\tensor\varphi^{tw^n}),N)=0$ for all $N$ and all $i>0$. That is, ${\rm coker}(1\tensor\varphi^{tw^n})$ is graded free. As noted in the proof of Proposition \ref{minRes}, 
${\rm coker}(1\tensor\varphi^{tw^n})\isom M^{\overline{tw}^n}$. Since $M$ is free if and only if $M^{\overline{tw}}$ is, $M$ is graded free. By Lemma \ref{trivial}, $(\varphi,\tau)\isom 0$ in $hTMF_A(f)$.

That $\mathcal F$ is faithful now follows from the triangulated structure (see Theorem 3.9 of \cite{Orlov1}).
\end{proof}

Zhang proves in \cite{Zhang} that, among other properties, being
noetherian, AS regular or AS-Gorenstein is invariant under graded
Morita equivalence. Thus, in the case where $A$ is left noetherian and
AS-regular, the equivalence theorems of Section \ref{section:zhang}
imply equivalences of the corresponding categories of singularities.

\section{Examples}
\label{section:examples}

\begin{example}
\label{Heisenberg}
Let $V$ be a finite-dimensional vector space over a field $k$ with
skew-symmetric, nondegenerate form $\omega$. Assume $\dim V=2n\ge 4$
and let $\mathfrak h$ be the corresponding Heisenberg Lie algebra and
$U(\mathfrak h)$ its universal enveloping algebra. Then $U(\mathfrak
h)$ can be presented by generators $x_1,\ldots, x_n, y_1,\ldots, y_n$
subject to the relations
$$[x_i,x_j]=[y_i,y_j]=0$$
$$ [x_i,y_j]=0\text{ for $i\neq j$}$$
$$[x_1,y_1]=[x_2,y_2]=\cdots=[x_n,y_n]$$
Since $\mathfrak h$ is a finite-dimensional Lie algebra, $U(\mathfrak
h)$ is Artin-Schelter regular of dimension $2n+1$. The element
$f=[x_1,y_1]$ is central and regular, and $B=U(\mathfrak h)/(f)\isom
k[x_1,\ldots, x_n, y_1,\ldots, y_n]$ is a commutative polynomial
ring. By Hilbert's Syzygy Theorem, every finitely generated left
$B$-module has a finite minimal graded free resolution. Thus there
exist no nontrivial reduced twisted left matrix factorizations of $f$.
\end{example}

\begin{example}
\label{ex:ore}
Let $A=k[x,y][w;\zeta]$ be a graded Ore extension of a commutative
polynomial ring in two variables by a graded automorphism $\zeta$,
where $wg=\zeta(g)w$ for all $g\in k[x,y]$.  Then $w^2$ is regular and
its normalizing automorphism is $\sigma=\zeta^{-2}$.  After choosing bases,
we define homomorphisms $\varphi:F\rightarrow G$ and $\tau:G^{\sigma}\rightarrow
F$ of graded free left $A$-modules via right multiplication by the matrices
$$[\varphi] = \begin{pmatrix} w& -\zeta(x)\\ 0&w\\ \end{pmatrix}\qquad\text{and}\qquad [\tau]=\begin{pmatrix}w & \zeta^2(x)\\ 0&w\\ \end{pmatrix}$$ 
Note $[\varphi^{\sigma}]=\begin{pmatrix}w&-\zeta^3(x)\\0&w\\ \end{pmatrix}$.
A straightforward verification shows that $\varphi\tau=\lambda_{w^2}$ and
$\tau\varphi^{\sigma}=\lambda_{w^2}$. (We remind the reader that
since we work with left modules, the composition is computed by
multiplying matrices in the opposite order.) Examining for
periodicity, we see that the minimal resolution
$\mathbf{\Omega}(\varphi,\tau)$ is periodic of period $p$ if and only
if $\zeta^p(x)=cx$ for some integer $p$ and scalar $c$. This example
suggests a useful method for constructing twisted factorizations with
desired properties. For example, if $\zeta(x)=x+y$ and $\zeta(y)=qy$
where $q$ is a primitive $n$-th root of unity, then the resolution is
periodic of period $n$.

As another example, taking $\zeta(x)=(x+y)/2$ and $\zeta(y)=y/2$ we
obtain a resolution which is not periodic. But it is interesting to
note that since $\zeta^n(x)\rightarrow 0$ as $n\rightarrow \infty$,
the limiting matrix $\begin{pmatrix} w & 0\\ 0 & w\\ \end{pmatrix}$
defines a minimal resolution which is periodic of period 1. In this
sense, the resolution becomes periodic after infinitely many steps.

In any case, extend $\zeta$ to a graded automorphism of $A$ by
$\zeta(w)=w$. Let $Z=\{\zeta^n~|~n\in\Z\}$ be the associated twisting
system. The twisted multiplication in $A^Z$ gives
$$g\ast w^2 = \zeta^{2}(g)w^2= w^2g = w^2\ast g$$
for all $g\in A$ so $w$ is central in $A^{Z}$. By Proposition
\ref{periodicity}, every twisted matrix factorization of $w^2$ over
$A^Z$ gives rise to a minimal graded free resolution of period at most
2.

\end{example}

\begin{example}
\label{Sklyanin}
Let $A=k\la x, y, z\ra/\la r_1, r_2, r_3\ra$ where 
\begin{align*}
r_1 &= yz+zy-x^2\\ 
r_2 &= xz+zx-y^2\\ 
r_3 &= xy+yx-z^2
\end{align*}

The algebra $A$ is a nondegenerate 3-dimensional Sklyanin algebra. The
element $g=2*(y^3+xyz-yxz-x^3)$ (the factor of 2 is only to clean up the twisted matrix factorization)
is central and regular in $A$, so $\sigma={\rm id}_A$. Let

$$\varphi = \left(\begin{smallmatrix} x & y & z & 0\\ -yz-2x^2 & -yx & zx-xz & x\\ xy-2yx & xz & -x^2 & y\\ -y^2-zx & x^2 & -xy & z\\ \end{smallmatrix}\right)
$$
and 
$$\tau=\left(\begin{smallmatrix} -zy & -x & z & y\\zx-xz & z & -y & x\\ xy & y & x & -z\\ 2xyz-4x^3 & -2x^2 & 2y^2 & 2(xy-yx)\\ \end{smallmatrix}\right)$$
be matrices with entries in $A$. One can check (it is not trivial) that $\tau\varphi=\varphi\tau=gI_4$. This matrix factorization produces a minimal resolution of the second syzygy module in a minimal resolution of the trivial module $_Bk$. Indeed if we put
$$M_2=\left(\begin{smallmatrix} -x & z & y\\ z & -y & x &\\ y & x & -z\\ -2x^2 & 2y^2 & 2(xy-yx)\\ \end{smallmatrix}\right)\qquad M_1=\left(\begin{smallmatrix} x\\ y\\ z\\ \end{smallmatrix}\right)$$
then 
$$\cdots \xrightarrow{\overline{\varphi}} B(-5)^3\oplus B(-6)\xrightarrow{\overline{\tau}} B(-3)\oplus B(-4)^3\xrightarrow{\overline{\varphi}}$$ $$B(-2)^3\oplus B(-3)\xrightarrow{M_2} B(-1)^3\xrightarrow{M_1} B$$
is a minimal graded free left $B$-module resolution of $_Bk$. 

\end{example}

\begin{example}
  Let $A=k_{q}[x,y]$ be the skew polynomial ring where $yx=qxy$ for
  some fixed $q\in k^{\times}$. Let $g$ be the graded automorphism of
  $A$ given by $g(x)=\lambda x$ and $g(y)=\lambda^{-1}y$ where
  $\lambda$ is a primitive $n$-th root of unity.  Let $G = \langle g
  \rangle$, the cyclic group of order $n$, act on $A$ with invariant
  subring $A^{G}$.  Classically (when $q=1$), this is an $A_n$
  Kleinian singularity.  It is not hard to check that $A^G$ is
  generated by $X:=x^n$, $Y:=xy$, and $Z:= y^n$, and $A^G \cong
  C/(\omega)$, where
$$C=k\la X,Y,Z\ra/\la YX- q^{n}XY, ZX - q^{n^2}XZ, ZY- q^{n}YZ\ra$$
is a skew polynomial ring and $\omega:= XZ-q^{-{\binom{n}{2}}}Y^n$ is a regular normal element of $C$ \cite[Case 2,2]{KKZ}.
Let $C$ be graded by setting $\deg X = \deg Z = n$, and $\deg Y = 2$.  Note that $C$
is noetherian and AS-regular of dimension 3 and one has relations
$$\omega X = q^{n^2} X\omega , \;\; \omega Y = Y \omega, \;\;\text{and}\;\; \omega Z = q^{-n^2} Z \omega.$$
The sets $M_j= \{a \in A~|~ g(a) = \lambda^j a\}$ for $0\le j<n$ are graded left $R=A^G$ modules, generated by $x^j$ and $y^{n-j}$. Note that $M_0=R$, and henceforth assume $j\neq 0$.

As a module over $C$, a minimal resolution of $M_j$ has the form
$$0 \longrightarrow C(-2n+j)\oplus C(-n-j) \stackrel{G_j}{\longrightarrow} C(-j) \oplus C(-n+j) \stackrel{D}{\longrightarrow} M_j \longrightarrow 0,$$
where the maps are given by right multiplication by\footnote{We adopt the usual convention that $\binom{k}{l} = 0$ for $k < l$.}:
$$ D:=
\begin{pmatrix}
x^j\\
y^{n-j}
\end{pmatrix}\text{ and }
G_j := \begin{pmatrix}
-q^{-{\binom{n-j}{2}}}Y^{n-j} & q^{(n-j)j}X\\
-Z & q^{nj-{\binom{j}{2}}} Y^j
\end{pmatrix}$$
Thus $\pd_C M_j = 1$, and hence $M_j$ is a maximal Cohen-Macaulay $R$-module. 
It is worth noting that when $q=1$, the $M_j$ form a complete set of maximal Cohen-Macaulay $R$-modules \cite[Example 5.25]{LW}.

Next, observe that
$$G_{n-j}G_j = \begin{pmatrix} -q^{(n-j)j} \omega & 0\\ 0 & -q^{(n-j)j+n^2} \omega \end{pmatrix} = G_j G_{n-j}$$
This shows $0\rightarrow
M_{n-j}(-n)\xrightarrow{\overline{G}_j}R(-j)\oplus
R(-n+j)\xrightarrow{D} M_j\rightarrow 0$, where $\overline{G}_j$ is
the $R$-module map induced on ${\rm coker}\ G_{n-j}$ by $G_j$, is an
exact sequence of $R$-modules. So a minimal graded $R$-module
resolution of $M_j$ is periodic of period at most 2 for every $0<
j<n$. (When $n=2$, the resolution has period 1.) With a small
adjustment, we obtain a complex arising from a twisted matrix
factorization of $\omega$.  Let
$$\Delta:= \begin{pmatrix}
-1 & 0\\
0 & -q^{n^2}
\end{pmatrix},
\quad N_j := G_j \Delta^{-1} = \begin{pmatrix}
q^{-{\binom{n-j}{2}}}Y^{n-j} & -q^{(n-j)j-n^2}X\\
 Z & -q^{nj -{\binom{j}{2}} - n^2} Y^j
\end{pmatrix},$$
and $P_{n-j} := q^{-(n-j)j}G_{n-j}$.  Then we have
$$P_{n-j} := \begin{pmatrix}
-q^{-{\binom{j}{2}} - j(n-j)}Y^{j} & X\\
-q^{-(n-j)j}Z & q^{(n-j)^2- {\binom{n-j}{2}}} Y^{n-j}
\end{pmatrix},$$ and
$$N_j^\sigma = \begin{pmatrix}
q^{-{\binom{n-j}{2}}}Y^{n-j} & -q^{(n-j)j}X\\
 q^{-n^2} Z & -q^{nj -{\binom{j}{2}} - n^2} Y^j
\end{pmatrix}.$$ Finally we have
$P_{n-j} N_j = \omega I = N_j^\sigma P_{n-j}$ as desired.
We note that $|\sigma| = |q|$, which can be an arbitrary positive integer or infinite.
\end{example}

\bibliographystyle{amsplain}
\bibliography{bibliog2}

\end{document}